\documentclass[a4paper,10pt]{article}

\usepackage[all]{xy}

\usepackage{amsmath}
\usepackage{amssymb}
\usepackage{amsthm}

\usepackage[normalem]{ulem}

\usepackage[pdftex]{graphicx} 
\newcommand{\sba}[1][n]{\ensuremath{b^{\prime}_{#1}}}
\newcommand{\sbaA}[1][n]{\ensuremath{b^{\prime,\mathcal{A}}_{#1}}}
\newcommand{\sbak}[1][n]{\ensuremath{b^{\prime,k}_{#1}}}
\newcommand{\Zatow}{\ensuremath{\mathbb{Z}[a^{\pm 1},b^{\pm 1},c^{\pm 1},d^{\pm 1} ,x^{\pm 1},y^{\pm 1},z^{\pm 1},w^{\pm 1}]}}

\newcommand{\iden}{{\rm 1\hspace*{-0.4ex}%
\rule{0.1ex}{1.52ex}\hspace*{0.2ex}}}

\DeclareMathOperator\TL{TL}
\DeclareMathOperator\End{End}

\DeclareMathOperator\op{op}

\makeatletter
\newtheorem*{rep@theorem}{\rep@title}
\newcommand{\newreptheorem}[2]{%
\newenvironment{rep#1}[1]{%
 \def\rep@title{#2 \ref{##1}}%
 \begin{rep@theorem}}%
 {\end{rep@theorem}}}
\makeatother

\theoremstyle{plain} \newtheorem{theorem}{Theorem}[section]
							       \newreptheorem{theorem}{Theorem}
                     \newtheorem*{introthm}{Theorem}
                     \newtheorem{prop}[theorem]{Proposition}
                     \newtheorem{lemma}[theorem]{Lemma}
                     
\theoremstyle{definition} \newtheorem{define}[theorem]{Definition}
                          \newtheorem{remark}[theorem]{Remark}

\title{A Tensor Space Representation of \\ the Symplectic Blob Algebra}
\author{Andrew Reeves}

\begin{document}

\maketitle

\begin{abstract}
The symplectic blob algebras $\{ \sba \}$ are a family of finite dimensional noncommutative algebras over $\mathbb{Z}[X_1,X_2,X_3,X_4,X_5,X_6]$ that can be defined in terms of planar diagrams in a way that extends the Temperley-Lieb and (ordinary) blob algebras.

In this paper we construct a new ``tensor space'' representation of the symplectic blob algebra $\mathcal{A} \otimes \sba$ for each $n \in \mathbb{N}$, for $\mathcal{A}$ a particular commutative ring with indeterminates.  The form of this representation is motivated by the XXZ representation of the Temperley-Lieb algebra $\TL_n$ \cite{jimbo86} and the related Martin-Woodcock representation of the blob algebra $b_n$ \cite{martinwoodcock2003}.  For $k$ an algebraically closed field, and for any $(\delta,\delta_L,\delta_R,\kappa_L,\kappa_R,\kappa) \in k^{6}$, the algebra $\sba$ specialises to a $k$-algebra $\sba(\delta,\delta_L,\delta_R,\kappa_L,\kappa_R,\kappa)$.  For any such specialisation, our representation passes to a $\sba(\delta,\delta_L,\delta_R,\kappa_L,\kappa_R,\kappa)$-module  $\mathcal{V}(n)$. 

In a second paper, we will show that this module $\mathcal{V}(n)$ is full-tilting (and hence faithful) whenever $\sba(\delta,\delta_L,\delta_R,\kappa_L,\kappa_R,\kappa)$ is quasihereditary.  This paper focuses on establishing that $\mathcal{V}(n)$ exists.
\end{abstract}

\section{Introduction}

Let $k$ be an algebraically closed field of characteristic $0$, and (for some $n \in \mathbb{N}$) let $\mathbb{K} := \mathbb{Z}[X_1, \ldots, X_n]$.  In order to understand the representation theory of algebras over $k$, it can often be useful to consider related algebras defined over $\mathbb{K}$. In particular, suppose that we are given some algebra $A$ over $\mathbb{K}$.  Every map $\iota : \{ X_1, \ldots, X_n \} \rightarrow k$ extends to an algebra morphism $\mathbb{K} \rightarrow k$, and thus gives $k$ the structure of a $\mathbb{K}$-algebra and a (right) $\mathbb{K}$-module.  So, for each such map, we can define a specialisation of $A$: 
\begin{displaymath}
 A^{\prime} := k \otimes_{\mathbb{K}} A \rm{,}
\end{displaymath}
which is a $k$-algebra.  (Of course, not all these specialisations will necessarily be unique.)

We can therefore associate to the original algebra $A$ an entire \emph{set} of algebras (over $k$), each indexed by an $n$-tuple of parameters in $k$.

In this paper we will consider a particular family of such (sets of) algebras: the symplectic blob algebras $\{ \sba \}$ \cite{greenmartinparker2007}.  These are algebras over $\mathbb{Z}[X_1, \ldots, X_6]$ that arise in statistical mechanics as extensions of the Temperley-Lieb and (ordinary) blob algebras (see, for example, \cite{degiernichols2009}; for the connections between diagram algebras and statistical mechanics in general see \cite{baxter} or \cite{martin2008}).  Almost all specialisations of the symplectic blob algebra to $k$ are semisimple and so, in at least some sense, fully understood, but understanding the non-generic representation theory remains an open problem.

The present paper is the first of a two-part series.  In this paper, for $\mathcal{A}$ a commutative ring which we shortly define, we construct a tensor space representation of each $\sba$ acting on an $\mathcal{A}$-module $V^{\otimes 4n}$.  We will show that for any specialisation $\sba(\uline{\Pi}) := k \otimes_{\mathcal{A}} \sba$ this representation passes to a $\sba(\uline{\Pi})$-module $\mathcal{V}(n)$.  In the sequel, we will focus on quasihereditary specialisations, and show that these modules are, in fact, full-tilting.

The main result of this paper is:

\begin{reptheorem}{thm:rep}
Fix $n \in \mathbb{N}$.  Let $\mathcal{A} = \Zatow{}$ and $\mathcal{Z} = \mathbb{Z}[D, D_L, D_R, K_L, K_R, K]$.  Fix an evaluation map $\theta : \mathcal{Z} \rightarrow \mathcal{A}$ that makes $\mathcal{A}$ a $\mathcal{Z}$-algebra.

Let $\mathcal{G}_n = \{ e, U_1, \ldots, U_{n-1}, f \}$ be the generators of $\sba$ given in Definition \ref{def:sbalgebra}.  Let $I_{n} = \{-2n+1,-2n+2\ldots,2n\}$.  Let $\{R^q_i\}_{i \in I_{2n}} \subset \End_{\mathcal{A}} (V^{\otimes4n})$ be the operators defined in Definition \ref{def:Rmatrices}.

Define a map $\mathcal{R} : \mathcal{G}_n \rightarrow \End_{\mathcal{A}}(V^{\otimes 4n})$ by 
 \begin{align}
  \mathcal{R}(U_i) &= R^a_{-n-i} R^b_{-n+i} R^c_{n-i} R^d_{n+i} \\
  \mathcal{R}(e) &= R^x_{-n} R^y_{n} \\
  \mathcal{R}(f) &= R^z_0 R^w_{2n} \rm{ .}  
 \end{align}
Then $\mathcal{R}$ extends to a unique representation of $\sbaA$, also called $\mathcal{R}$, if and only if the action of $\mathcal{Z}$ on $\mathcal{A}$ is such that:
 \begin{align}
  \theta(D) &= \left( a + \frac{1}{a} \right) \left( b + \frac{1}{b} \right) \left( c + \frac{1}{c} \right) \left( d + \frac{1}{d} \right)  \\
  \theta(D_L) &= \left( x + \frac{1}{x} \right) \left( y + \frac{1}{y} \right) \\
  \theta(D_R) &= \left( z + \frac{1}{z} \right) \left( w + \frac{1}{w} \right) \\
  \theta(K_L) &= \left( \frac{ab}{x} + \frac{x}{ab} \right) \left( \frac{cd}{y} + \frac{y}{cd} \right) \\
  \theta(K_R) &= \left( \frac{ad}{w} + \frac{w}{ad} \right) \left( \frac{bc}{z} + \frac{z}{bc} \right) \\
  \theta(K) &= \left\{ \begin{array}{lr}
                    \frac{xy}{zw} + 2 + \frac{zw}{xy} & \mbox{if } n \mbox{ is odd} \\
                    \frac{abcd}{xyzw} + 2 + \frac{xyzw}{abcd} & \mbox{if } n \mbox{ is even}
                   \end{array} \right. \rm{ .} 
 \end{align}
\end{reptheorem}

\subsection*{The XXZ Representation of $\TL_n$}

The motivation for our construction is the following result:

Let $V_2$ be the free $\mathbb{Z}[q,q^{-1}]$ module with basis $\{v_1, v_2\}$.  Then the $n$-fold tensor product
\begin{displaymath}
 V_2^{\otimes n} := \underbrace{V_2 \otimes_{\mathbb{Z}[q,q^{-1}]} \ldots \otimes_{\mathbb{Z}[q,q^{-1}]} V_2}_{\mbox{n times}}
\end{displaymath}
has basis $\mathcal{B}_n = \{v_{i_1} \ldots v_{i_n} \; | \; i_j \in \{ 1, 2 \}\}$.  Recall that the Temperley-Lieb algebra $\TL_n(q+q^{-1})$ over $\mathbb{Z}[q,q^{-1}]$ has generators $U_1, U_2, \ldots U_{n-1}$ satisfying the relations:
\begin{align}
U_j^2 &= \left( q + q^{-1} \right) U_j &\qquad \mbox{for all } j \\
U_j U_k U_j &= U_j &\qquad \mbox{if } |j-k| = 1 \\
U_j U_k &= U_k U_j &\qquad \mbox{if } |j-k| > 1 \rm{.}
\end{align}

There is an action of these generators $\{U_1, \ldots, U_{n-1}\}$ on $V_2^{\otimes n}$ which extends to a representation of the algebra.
\begin{introthm}
 For $j \in \{1, \ldots (n-1)\}$ and $v = v_{i_1} \ldots v_{i_n} \in \mathcal{B}_n$, define $U_j \circ v$ by
\begin{displaymath}
 U_j \circ v = \left\{ \begin{array}{cl}
                        q^{2 - i_{j+1}} v_{i_1} \ldots v_{i_{j-1}} v_{1} v_2 v_{i_{j+2}} \ldots v_{i_{n}} & \\
                         \qquad \qquad + q^{1 - i_{j}} v_{i_1} \ldots v_{i_{j-1}} v_{1} v_2 v_{i_{j+2}} \ldots v_{i_{n}} & \mbox{if } v_i \neq v_{i+1} \\
                        0 & \mbox{if } v_i = v_{i+1}
                       \end{array} \right.  \rm{.}
\end{displaymath}
This action extends to a unique representation
\begin{displaymath}
 \mathcal{J} : \TL_n(q+q^{-1}) \rightarrow \End(V_2^{\otimes n})  \rm{.}
\end{displaymath} \qed \end{introthm}

This representation appears as a special case of Theorem 1 of \cite{jimbo86}.  It was previously an object of study in statistical mechanics: see, for example, \cite{baxter, temperleylieb71}.  It is sometimes known as the XXZ representation.  The XXZ representation is faithful and (for quasihereditary specialisations of the algebra) full-tilting \cite{martin2003, martin92, martinryomhansen2004}.

Motivated by this construction, in this paper we construct a tensor space representation $\mathcal{R}$ of the symplectic blob algebra $\sba$ over the commutative ring $\mathcal{A} = \Zatow$. 
For $V$ a free $\mathcal{A}$-module of rank $2$, and for any $n \in \mathbb{N}$, we give an action of $\sba = \sba(X_1, X_2, X_3, X_4, X_5, X_6)$ on an $\mathcal{A}$-module $V^{\otimes 4n} := V \otimes_{\mathcal{A}} \ldots \otimes_{\mathcal{A}} V$, for $\delta$ ,$\delta_L$,$\delta_R$,$\kappa_L$,$\kappa_R$,$\kappa \in \mathcal{A}$.  The form of this representation is suggested by the XXZ representation $\mathcal{J}$ of the Temperley-Lieb algebra $TL_n$, discussed above, and by an `unfolding' map given in \cite{greenmartinparker2007},

The representation $\mathcal{R}$ has several properties that are not shared by any other known tensor space representation of the algebra $\sba$.  Such representations arise naturally in statistical mechanics; examples include the (extended) Potts and XXZ representations (see, for example, \cite{nichols2006,degiernichols2009}).  However, the Potts representations are defined only for particular specialisations of the algebra parameters $\delta,\delta_L,\delta_R,\kappa_L,\kappa_R,\kappa$ and neither they nor the XXZ representations are faithful.  In contrast, the representation $\mathcal{R}$ is defined over a ring $\mathcal{A}$ which allows for base changes to many different specialisations. These properties suggest that $\mathcal{R}$ may be specialised to a full-tilting module over any field $k$ that is also an $\mathcal{A}$-algebra.  In a second paper we will show that this can in fact be done.   

Constructing such a module is the first step in extending the 'virtual algebraic Lie theory' programme, begun for the (ordinary) blob algebra $b_n$ in \cite{martinryomhansen2004}, to the symplectic blob algebra.  The endomorphism algebra of a full-tilting module, called the Ringel dual, provides significant information about the representation theory of the original algebra \cite{dengduparshallwang}.  The construction of a Ringel dual would therefore provide a useful tool for further study of the non-generic representation theory of the symplectic blob algebra, which is presently not well understood \cite{greenmartinparker2008}.  

The focus of the present paper is to establish that the representation $\mathcal{R}$ can be constructed for all $n \in \mathbb{N}$, and that it passes to a representation $\mathcal{R}_{\uline{\Sigma}}$ of any specialisation $\sbak(\delta, \delta_L, \delta_R, \kappa_L, \kappa_R, \kappa) := k \otimes_{\mathcal{Z}} \sba$.   

\subsection{Overview}

In section 2 we define the symplectic blob algebra $\sba$ over a commutative ring $\mathcal{Z}$ (defined below) and discuss the 'unfolding' map of \cite{greenmartinparker2007} in some more detail.  The main result of this section is Theorem \ref{thm:rep}:  For $\mathcal{A}$ a particular commutative ring, we define an $\mathcal{A}$-module $V$, and use this to construct a module $V^{\otimes 4n}$ for the symplectic blob algebra $\sbaA := \mathcal{A} \otimes_{\mathcal{Z}} \sba$, for each $n \in \mathbb{N}$.  

In section 3 we consider specialisations of the algebra over an algebraically closed field $k$.  For $\uline{\Pi} = (\delta, \delta_L, \delta_R, \kappa_L, \kappa_R, \kappa) \in k^{6}$ fixed but arbitrary, and $\theta : \mathcal{Z} \rightarrow k$ a ring homomorphism such that
\begin{displaymath}
 \left( \theta(X_1), \theta(X_2), \theta(X_3), \theta(X_4), \theta(X_5), \theta(X_6) \right) = \uline{\Pi} \rm{,}
\end{displaymath}
 making $\mathcal{Z}$ into a $k$-algebra, we define $\sbak(\uline{\Pi}) := k \otimes_{\mathcal{Z}} \sba$.  Our main result is Corollary \ref{corl:repexists}: the $\sbaA$-module $V^{\otimes 4n}$ passes to a $\sbak(\uline{\Pi})$-module $\mathcal{V}(n)$ for every $\uline{\Pi} \in k^6$.
 
\subsection{Notation}

Let $\mathbb{N} = \{0,1,2,\ldots\}$ be the usual natural numbers, let $k$ be an algebraically closed field of characteristic $0$, and define commutative rings $\mathcal{Z}$ and $\mathcal{A}$ by
\begin{displaymath}
 \mathcal{Z} := \mathbb{Z}[X_1, X_2, X_3, X_4, X_5, X_6] \rm{,}
\end{displaymath}
\begin{displaymath}
 \mathcal{A} := \Zatow \rm{.}
\end{displaymath}

We will write $\uline{\Pi}$ for the $6$-tuple $(\delta, \delta_L, \delta_R, \kappa_L, \kappa_R, \kappa) \in k^{6}$, and $\uline{\Sigma}$ for the $8$-tuple $(a_0, b_0, c_0, d_0, x_0, y_0, z_0, w_0) \in k^{8}$.

Let $\delta(x,y)$ be the Kronecker delta, so that
\begin{align*}
 \delta(x,y) &= \left\{ \begin{array}{lr} 
        1 & \text{if } x = y \\ 
        0 & \text{if } x \neq y \end{array} \right. \rm{,}
\end{align*}
and let $\delta`(x,y) = 1 -  \delta(x,y)$.

For any invertible $q$ in a commutative ring, and for any $n \in \mathbb{N} \setminus \{ 0 \}$, let $[n]_q = q^{n-1} + q^{n-3} + \ldots + q^{1-n}$.  In particular, $[2]_q = q + q^{-1}$.  (These are the $q$-numbers or Gaussian polynomials that arise frequently in the study of quantum groups.  See, for instance, \cite{kassel}.)

For $n \in \mathbb{N}$, let $\{1,2\}^n$ be the set of all finite sequences of length $n$ in the alphabet $\{1,2\}$.

Let $V$ be the free $\mathcal{A}$-module with basis $\{v_1, v_2\}$.  Then 
\begin{displaymath}
 V^{\otimes m} := \underbrace{V \otimes_{\mathcal{A}} \ldots \otimes_{\mathcal{A}} V}_{m \mbox{ times}}
\end{displaymath}
is also a free $\mathcal{A}$-module.  As discussed in Definition \ref{def:Vmodule}, below, we will identify a basis of $V^{\otimes m}$ with the set of sequences $\{ 1,2 \}^{m}$; if $w$ is such a sequence, then we write $\uline{w}$ for the corresponding element of $V^{\otimes m}$.

\section{A Representation over \\ $\Zatow$}

\subsection{Definitions}

In this section we define the symplectic blob algebra $\sba$ over the commutative ring $\mathbb{Z}[X_1, X_2, X_3, X_4, X_5, X_6]$ in terms of generators and relations.  This presentation of the algebra is shown in \cite{greenmartinparker2011} to be isomorphic to the original definition of the algebra in terms of decorated Temperley-Lieb diagrams \cite{greenmartinparker2007}.  

\begin{figure}
 \centering
 \includegraphics[scale=0.5]{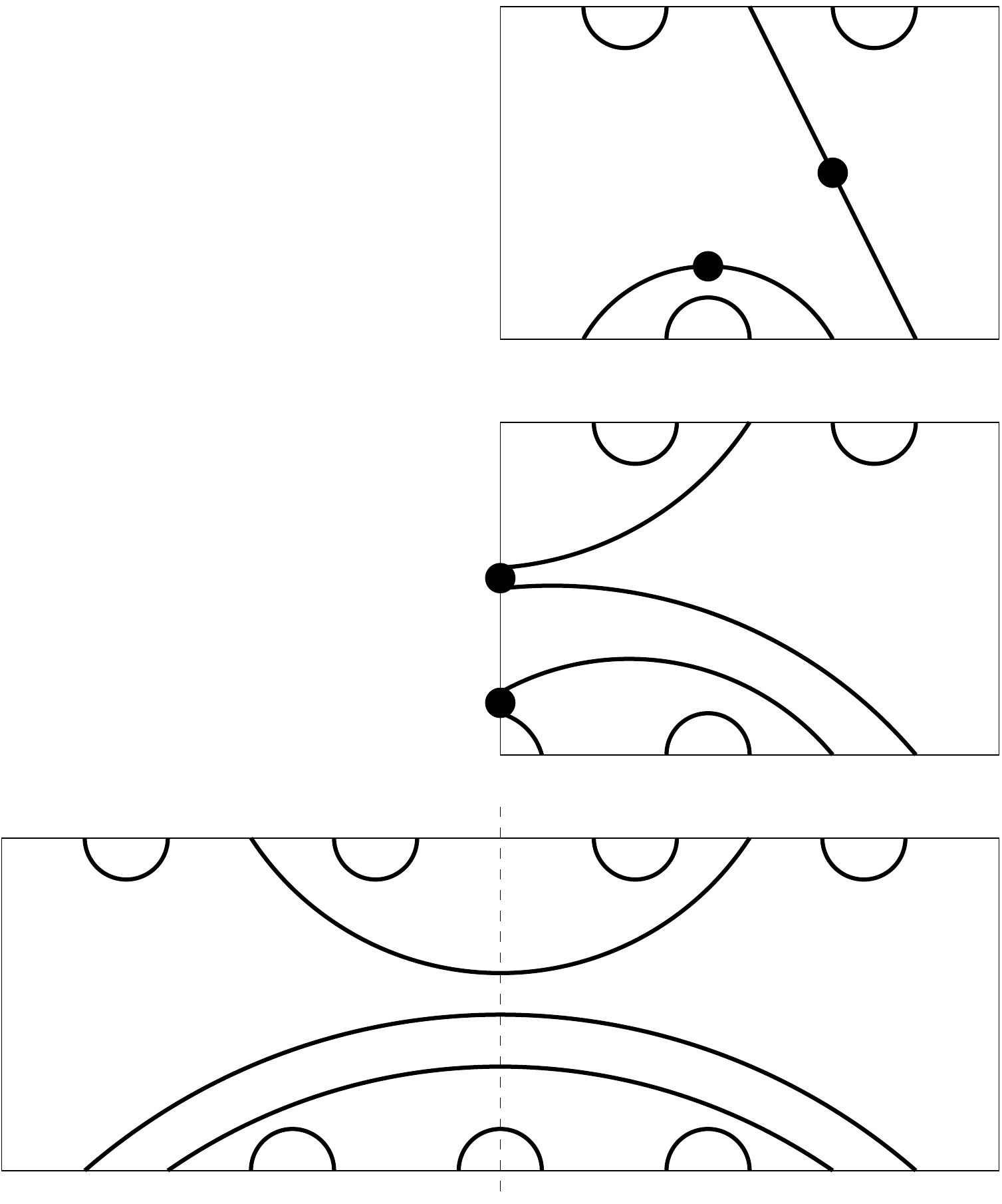}
 \caption{Unfolding A Blob Diagram}
 \label{fig:unfoldingblob}
\end{figure}

\begin{define}[The Symplectic Blob Algebra]  \label{def:sbalgebra}
Fix $n \in \mathbb{N}$. 

The symplectic blob algebra $\sba$ is the associative, unital $\mathcal{Z}$-algebra with generators $\mathcal{G}_n = \{ e, U_1, \ldots, U_{n-1}, f \}$ satisfying the relations below.
 \begin{align}
 U_i^2  &= X_1 U_i \qquad &\text{for all } i \label{rel:sba1} \\
 U_i U_j U_i &= U_i \qquad &\text{if } |i-j| = 1 \label{rel:sba9} \\ 
 U_i U_j &= U_j U_i \qquad &\text{if } |i-j| \neq 1 \label{rel:sba8} \\ 
 e^2 &= X_2 e \label{rel:sba2} \\
 f^2 &= X_3 f \label{rel:sba3} \\
 U_1 e U_1 &= X_4 U_1 \label{rel:sba4} \\
 U_{n-1} f U_{n-1} &= X_5 U_{n-1} \label{rel:sba5} \\
 e U_i &= U_i e \qquad &\text{if } i \neq 1 \label{rel:sba10} \\
 f U_i &= U_i f \qquad &\text{if } i \neq n-1 \label{rel:sba11} \\
 ef &= fe \qquad &\text{if } n > 1 \label{rel:sba12} \\
 IJI &= X_6 I \label{rel:sba6} \\
 JIJ &= X_6 J \label{rel:sba7}
\end{align}
where
\begin{equation*}
 I = \left\{ \begin{array}{lr}
              U_1 U_3 \ldots U_{n-2} f & \text{if } n \text{ is odd} \\
              U_1 U_3 \ldots U_{n-1} & \text{if } n \text{ is even}
             \end{array} \right.
\end{equation*}
and
\begin{equation*}
 J = \left\{ \begin{array}{lr}
              e U_2 \ldots U_{n-1} & \text{if } n \text{ is odd} \\
              e U_2 \ldots U_{n-2} f & \text{if } n \text{ is even}
             \end{array} \right. \text{.}
\end{equation*}
\end{define}

The Temperley-Lieb algebra $\TL_n(X_1)$ \cite{temperleylieb71} is isomorphic to the subalgebra of $\sba$ generated by $\{U_1,U_2, \ldots, U_{n-1}\}$ .  Similarly, the blob algebra $b_n(X_1,X_2,X_4)$ \cite{martinsaleur94} is isomorphic to the subalgebra of $\sba$ generated by $\{e,U_1,U_2, \ldots, U_{n-1}\}$.

\begin{remark}  \label{rmk:unfolding} These three algebras - $\TL_n$, $b_n$ and $\sba$ - are each isomorphic to a particular 'diagram algebra'  (as described, for example, in \cite{greenmartinparker2007}).  Such an algebra has a basis given by a fixed set of planar diagrams (with multiplication of two basis diagrams defined by concatenation,  followed by the application of one or more straightening rules).  This realisation of the symplectic blob algebra will not be the focus of the present paper.  However, it inspires the form of our candidate representation, as we now briefly discuss.  We assume that the reader has some familiarity with the diagram calculus of such algebras.

In \cite{martinwoodcock2003}, the authors construct a tensor space representation $\mathcal{T}$ of the blob algebra $b_n$. The construction of this representation is motivated by the XXZ representation $\mathcal{J}$ of the Temperley-Lieb algebra $\TL_n$ \cite{jimbo86} and by an ``unfolding'' map $\mu$.  This map sends basis diagrams of $b_n$ to basis diagrams of $\TL_{2n}$, as illustrated in Figure \ref{fig:unfoldingblob}.  However, it does not extend to an algebra homomorphism; $b_n$ depends on three parameters, and $\TL_{2n}$ on only one.  So the representation $\mathcal{J}$ does not lift to a representation of $b_n$; instead, the representation $\mathcal{T}$ is constructed so that, for any basis diagram $D \in b_n$, $\mathcal{T}(D)$ is mask equivalent to $\mathcal{J}(\mu(D))$ in the sense of Definition 1 of \cite{martin2003}.  The representation $\mathcal{T}$ can be shown to share many of the properties of $\mathcal{J}$ \cite{martinryomhansen2004}.

\begin{figure}
 \centering
 \includegraphics[scale=0.35]{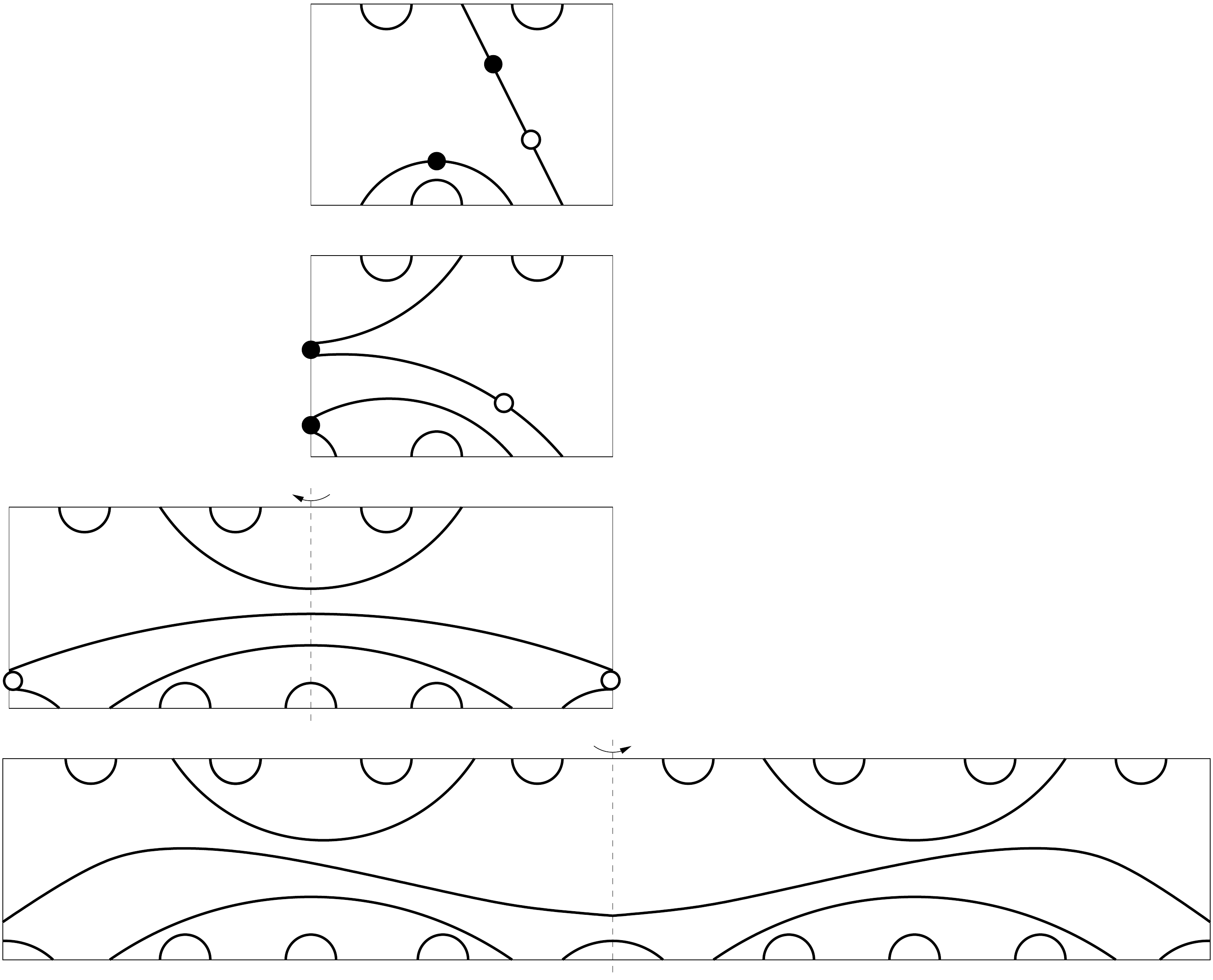}
 \caption{Unfolding A Symplectic Blob Diagram}
 \label{fig:unfoldingsymp}
\end{figure}

Symplectic blob diagrams can have two distinct type of decoration: visualised as either a solid or a hollow blob.  An unfolding map $\nu$, the natural generalisation of $\mu$, is defined in \cite{greenmartinparker2007}.  The map $\nu$ `unfolds' a symplectic blob diagram $S$ along first the left and then the right boundary, as in Figure \ref{fig:unfoldingsymp}.  This unfolding produces a \emph{periodic} Temperley-Lieb diagram \cite[Section 7]{greenmartinparker2007}.
\end{remark}

\subsection{Operators on $V^{\otimes m}$}

We now define a family of operators on $V^{\otimes 4n}$, where $V$ is a particular $\mathcal{A}$-module.  In the next section we will use these operators to construct a tensor space representation of $\sba$ for each $n \in \mathbb{N}$.

For any commutative ring $R$, let $R^{\times}$ denote the group of units.

Let $V$ be the free $\mathcal{A}$-module with basis $\{ v_1, v_2 \}$.  Then for any $m \in \mathbb{N}$, the $m$-fold tensor product
\begin{displaymath}
 V^{\otimes m} := \underbrace{V \otimes_{\mathcal{A}} \ldots \otimes_{\mathcal{A}} V}_{m \mbox{ times}} 
\end{displaymath}
has basis $\left\{ v_{i_1} \otimes \ldots \otimes v_{i_m} \; | \; i_1 \ldots i_m \in \{ 1,2 \}^{m}\right\}$.

For any $n \in \mathbb{N}$, let $I_{n} = \{ -2n+1, -2n+2, \ldots 2n \}$.

Index the factors of $V^{\otimes 4n}$ by $I_n$, letting $\alpha_{j-2n} := i_{j}$, so that a typical basis element is denoted $v_{\alpha_{-2n+1}} \otimes v_{\alpha_{-2n+2}} \otimes v_{\alpha_{2n}}$.
\begin{define}  \label{def:Vmodule}
Define a map $\uline{\;} : \{ 1,2 \}^{4n} \rightarrow V^{\otimes 4n}$ by
\begin{displaymath}
 \uline{\alpha_{-(2n-1)} \alpha_{-(2n-2)} \ldots \alpha_{2n}} := v_{\alpha_{-2n+1}} \otimes v_{\alpha_{-2n+2}} \otimes \ldots \otimes v_{\alpha_{2n}} \rm{.}
\end{displaymath}
\end{define}

We now define a family of operators on this module $V^{\otimes 4n}$.

\begin{define}  \label{def:Rmatrices}
Let $n \in \mathbb{N}$.   Let $q \in \mathcal{A}^{\times}$.

Define a family of operators $\{ R^{q}_i \}_{i \in I_{n}} \subset \End_{\mathcal{A}}( V^{\otimes 4n})$ as follows:

Let $\alpha = \uline{\alpha_{-2n+1} \ldots \alpha_{2n}} \in V^{\otimes 4n}$ be a basis element.  Then for $i \in I_{n} \setminus \{2n\}$ we define $R^{q}_{i}$ by 
\begin{align*}
 R^{q}_i \circ \alpha &= \delta'(\alpha_i,\alpha_{i+1}) \left( q^{2 - \alpha_i} \uline{\alpha_{-2n+1} \ldots 1 2 \ldots \alpha_{2n}} \right. \\
& \qquad \qquad \qquad \qquad \qquad + \left.  q^{1 - \alpha_i} \uline{\alpha_{-2n+1} \ldots 2 1 \ldots \alpha_{2n}} \right) \rm{,}
\end{align*}
and we define $R^{q}_{2n}$ by
\begin{align*} 
R^{q}_{2n} \circ \alpha &= \delta`(\alpha_{2n},\alpha_{-2n+1}) \left( q^{2 - \alpha_{2n}} \uline{2 \alpha_{-2n+2} \ldots \alpha_{2n-1} 1} \right. \\
& \qquad \qquad \qquad \qquad \qquad + \left. q^{1 - \alpha_{2n}} \uline{ 1 \alpha_{-2n+2} \ldots \alpha_{2n-1} 2 } \right) \rm{.}
\end{align*}
\end{define}

\begin{figure}
 \centering
 \includegraphics[scale=0.65]{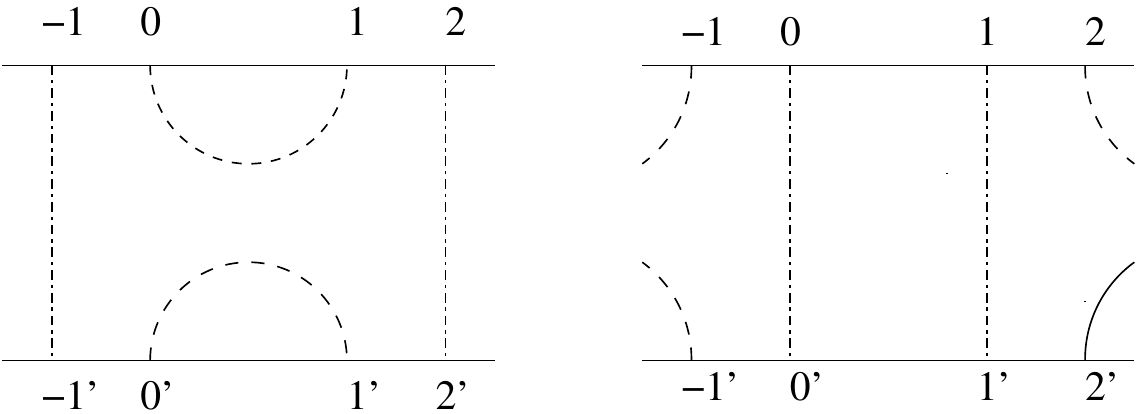}
 \caption{Index Diagrams for $R^{q}_{0}$ and $R^{q}_{2}$ in $\End_{\mathcal{A}}(V^{\otimes 4})$}
 \label{fig:indexdiagrams}
\end{figure}

In matrix notation, for $i \neq 2n$, note that
\begin{displaymath}
 R^{q}_i = \iden_{2^{i-1}} \otimes \begin{pmatrix}
                                    0 & 0 & 0 & \\
                                    0 & q & 1 & 0 \\
                                    0 & 1 & \frac{1}{q} & 0 \\
                                    0 & 0 & 0 & 0
                                   \end{pmatrix} \otimes \iden_{2^{n-1-i}} \rm{.}
\end{displaymath}
where $\iden_{m}$ is the $m$ by $m$ identity matrix.  

It will be useful to illustrate the action of these operators on $V^{\otimes 4n}$ graphically, as in Figure \ref{fig:indexdiagrams}.  Compare these with the ''chip diagrams'' used in \cite{doikoumartin2006}.  

Each factor of $V^{\otimes n}$ is represented by a labeled point on two parallel horizontal lines.  The identity matrix $\iden_{2}$ is represented by a vertical line joining the two points.  The matrix
\begin{displaymath}
\begin{pmatrix}
0 & 0 & 0 & \\
0 & q & 1 & 0 \\
0 & 1 & \frac{1}{q} & 0 \\
0 & 0 & 0 & 0
\end{pmatrix}
\end{displaymath}
is represented by a ''cup'' and ''cap'' joining two adjacent points on either line.

\subsection{The Representation $\mathcal{R}$}

The main result of this section is Theorem \ref{thm:rep}.

Pick a map $\theta : \{ D, D_L, D_R, K_L, K_R, K \} \rightarrow \mathcal{A}$, and extend this to a (unique) unital algebra homomorphism  $\mathcal{Z} \rightarrow \mathcal{A}$.  This makes $\mathcal{A}$ a $\mathcal{Z}$-algebra: $z \circ a := \theta(z) a$.  Define $\sbaA$ by
\begin{displaymath}
 \sbaA := \mathcal{A} \otimes_{\mathcal{Z}} B_n^{\prime} \rm{.}
\end{displaymath}

\begin{theorem} \label{thm:rep}
Fix $n \in \mathbb{N}$.  Let $\mathcal{A} = \Zatow{}$ and $\mathcal{Z} = \mathbb{Z}[D, D_L, D_R, K_L, K_R, K]$.  Fix a map $\theta : \mathcal{Z} \rightarrow \mathcal{A}$ that makes $\mathcal{A}$ a $\mathcal{Z}$-algebra, as above.

Let $\mathcal{G}_n = \{ e, U_1, \ldots, U_{n-1}, f \}$ be the generators of $\sba$ given in Definition \ref{def:sbalgebra}.  Let $I_{n} = \{-2n+1,-2n+2\ldots,2n\}$.  Let $\{R^q_i\}_{i \in I_{2n}} \subset \End_{\mathcal{A}} (V^{\otimes4n})$ be the operators defined in Definition \ref{def:Rmatrices}.

Define a map $\mathcal{R} : \mathcal{G}_n \rightarrow \End_{\mathcal{A}}(V^{\otimes 4n})$ by 
 \begin{align}
  \mathcal{R}(U_i) &= R^a_{-n-i} R^b_{-n+i} R^c_{n-i} R^d_{n+i} \label{tsr:eq1} \\
  \mathcal{R}(e) &= R^x_{-n} R^y_{n} \label{tsr:eq2} \\
  \mathcal{R}(f) &= R^z_0 R^w_{2n} \rm{ .} \label{tsr:eq3} 
 \end{align}
Then $\mathcal{R}$ extends to a unique representation of $\sbaA$, also called $\mathcal{R}$, if and only if the action of $\mathcal{Z}$ on $\mathcal{A}$ is such that:
 \begin{align}
  \theta(D) &= \left( a + \frac{1}{a} \right) \left( b + \frac{1}{b} \right) \left( c + \frac{1}{c} \right) \left( d + \frac{1}{d} \right) \label{thmcon1} \\
  \theta(D_L) &= \left( x + \frac{1}{x} \right) \left( y + \frac{1}{y} \right) \label{thmcon2} \\
  \theta(D_R) &= \left( z + \frac{1}{z} \right) \left( w + \frac{1}{w} \right) \label{thmcon3} \\
  \theta(K_L) &= \left( \frac{ab}{x} + \frac{x}{ab} \right) \left( \frac{cd}{y} + \frac{y}{cd} \right) \label{thmcon4} \\
  \theta(K_R) &= \left( \frac{ad}{w} + \frac{w}{ad} \right) \left( \frac{bc}{z} + \frac{z}{bc} \right) \label{thmcon5} \\
  \theta(K) &= \left\{ \begin{array}{lr}
                    \frac{xy}{zw} + 2 + \frac{zw}{xy} & \mbox{if } n \mbox{ is odd} \\
                    \frac{abcd}{xyzw} + 2 + \frac{xyzw}{abcd} & \mbox{if } n \mbox{ is even}
                   \end{array} \right. \rm{ .} \label{thmcon6}
 \end{align}
\end{theorem}

The heuristic idea is behind this choice of $\mathcal{R}$ is that if $D$ is a symplectic blob diagram realising the element $d \in \sba$, then the index diagram of $\mathcal{R}(d)$ should be of the same form as $\mu(D)$.  

Before proving this theorem we will need to record some computational lemmas.

\begin{lemma}  \label{lem:TLlike}
Let $n \in \mathbb{N}$ with $n \geq 1$.  Let $m \in I_n$ such that $m \neq 2n$.  Let $q \in \mathcal{A}^{\times}$.  Then
\begin{align}
\left( R^q_m \right)^2 &= [2]_q R^q_m \\
R^q_{m} R^{q}_{m \pm 1} R^{q}_{m} &= R^q_m \rm{.}
\end{align}
\end{lemma}
\begin{proof}
These identities are used to prove the existence of the XXZ representation of $\TL_n$ given in the introduction.  See, for example, \cite{baxter, martin92}. 
\end{proof}

\begin{lemma}  \label{lem:1}
Let $n \in \mathbb{N}$ with $n \geq 1$.  Let $m \in I_n$ such that $-2n+1 < m < 2n$.  Then for any $q,s,t \in \mathcal{A}^{\times}$:
\begin{align} 
 \left(R^s_{m+1} R^t_{m-1}\right) R^q_m \left(R^s_{m+1} R^t_{m-1}\right) = [2]_{\frac{q}{st}} R^s_{m+1} R^t_{m-1}  \rm{ .} \label{eq:lem}
\end{align}
\end{lemma}
\begin{proof}
 This result is essentially equivalent to calculations used in the proof of Proposition 6.4 in \cite{martinwoodcock2003}.
\end{proof}

The identity \eqref{eq:lem} is illustrated by the index diagram in Figure \ref{fig:indexlemma1}. 

\begin{figure}
 \centering
 \scalebox{0.65}{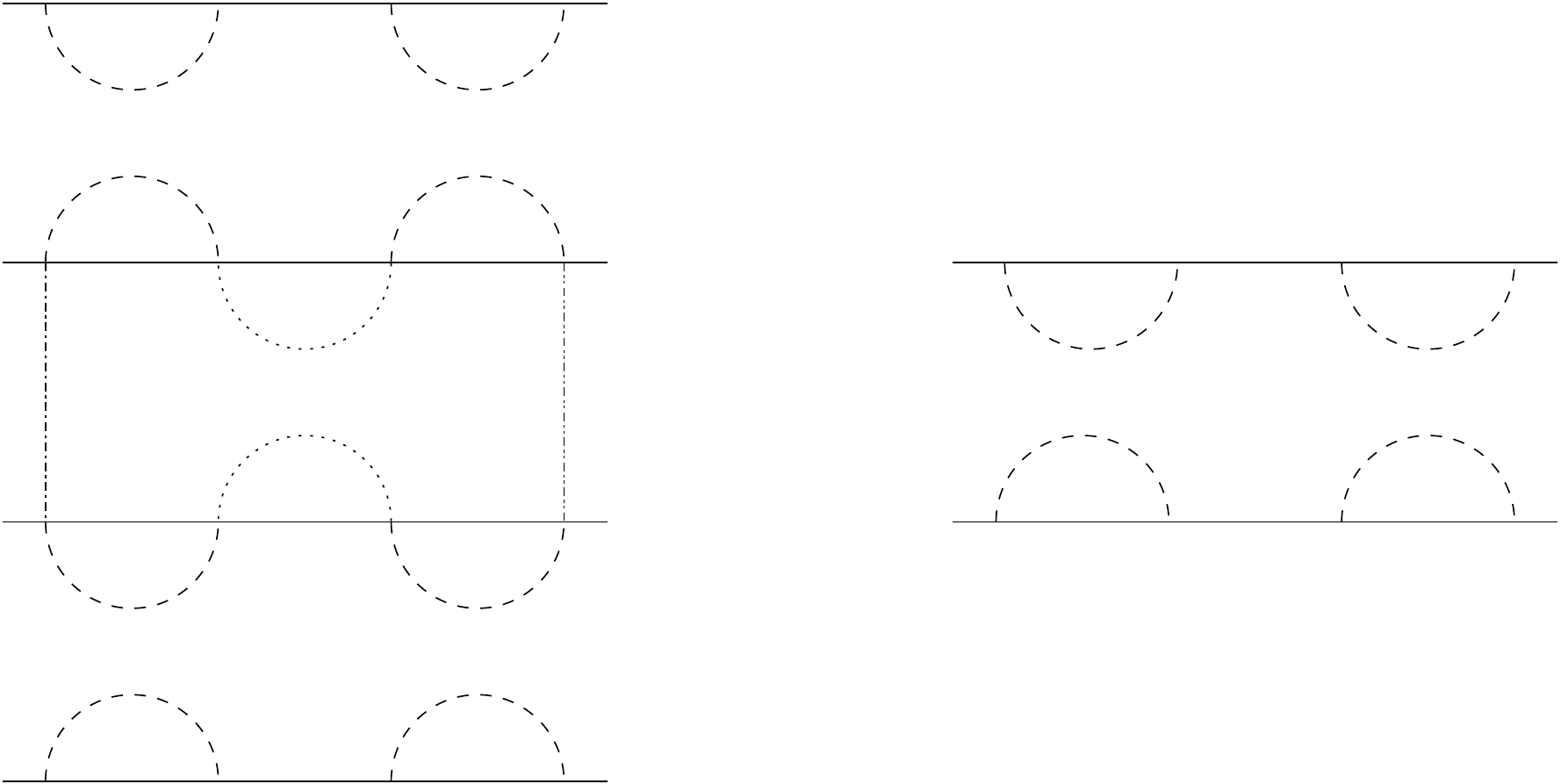}
 \caption{$\left(R^s_{1} R^t_{1}\right) R^q_0 \left(R^s_{1} R^t_{-1}\right) = [2]_{\frac{q}{st}} R^s_{1} R^t_{-1}$}
 \label{fig:indexlemma1}
\end{figure}

\begin{lemma}  \label{lem:IJ}
Let $\{q_i\}_{i=-2n+1}^{2n} \subset \mathcal{A}^{\times}$. Let $R_i = R^{q_i}_i \in \End(V^{\otimes 4n})$.  Define $\mathcal{O}, \mathcal{E} \in$ $\End_{\mathcal{A}}(V^{\otimes 4n})$ by
\begin{align*}
 \mathcal{O} &= R_{-2n+1} R_{-2n+3} \ldots R_{2n-1} \\
 \mathcal{E} &= R_{-2n+2} R_{-2n+4} \ldots R_{2n} \rm{.}
\end{align*}
Let $Q \in \mathcal{A}$ be given by
\begin{equation*}
 Q = \frac{q_{-2n+1} q_{-2n+3} \ldots q_{2n-1}}{q_{-2n+2} q_{-2n+4} \ldots q_{2n}} + 2 + \frac{q_{-2n+2} q_{-2n+4} \ldots q_{2n}}{q_{-2n+1} q_{-2n+3} \ldots q_{2n-1}} \rm{.}
\end{equation*}
Then
\begin{align*}
 \mathcal{E} \mathcal{O} \mathcal{E} &= Q \mathcal{E} \\
 \mathcal{O} \mathcal{E} \mathcal{O} &= Q \mathcal{O} \rm{.}
\end{align*}
\end{lemma}
\begin{proof}
Let $\mathcal{S}_O$ be the set of basis elements in $V^{\otimes 4n}$ which are not killed by $\mathcal{O}$.  Similarly let $\mathcal{S}_E$ be the set of basis elements which are not killed by $\mathcal{E}$.

We have:
\begin{align*}
 \mathcal{S}_O &= \{ a_{-2n-1} \ldots a_{2n} \in V^{\otimes 4n} \; | \; \delta'(\alpha_{2i},\alpha_{2i-1}) = 1 \mbox{ for } i \in \{ -(n-1), \ldots, n \} \} \\
 \mathcal{S}_E &= \{ a_{-2n-1} \ldots a_{2n} \in V^{\otimes 4n} \; | \; \delta'(a_{2i},a_{2i+1}) = 1 \mbox{ for } i \in \{-(n-1), \ldots, n \} \} \rm{,}
\end{align*}
(adopting the convention that $\alpha_{2n+1} := \alpha_{-2n+1}$) and so
\begin{align*}
 \mathcal{S}_O \cap \mathcal{S}_E = \{ \uline{1212\ldots12}, \uline{2121\ldots21} \}   \rm{.}
\end{align*}
Let $\alpha$ be an arbitrary element of $\mathcal{S}_O$.  Then
\begin{align*}
 \mathcal{O} \circ \alpha &= \sum_{v \in \mathcal{S}_O} \left( \prod_{j=-(n-1)}^{n} q_{2j-1}^{3 - \alpha_{2j-1} - v_{2j-1}} \right) v  \rm{,}
\end{align*}
and so in particular
\begin{align*}
 \mathcal{O} \circ \uline{1212\ldots12} &= \sum_{v \in \mathcal{S}_O} \left( \prod_{j=-(n-1)}^{n} q_{2j-1}^{2 - v_{2j-1}} \right) v \rm{,} \\
\mathcal{O} \circ  \uline{2121\ldots21} &= \sum_{v \in \mathcal{S}_O} \left( \prod_{j=-(n-1)}^{n} q_{2j-1}^{1 - v_{2j-1}} \right) v \rm{.}
\end{align*}

Now let $\beta$ be an arbitrary element of $\mathcal{S}_E$.  Then
\begin{align}
 \mathcal{E} \circ \beta &= \sum_{w \in \mathcal{S}_E} \left( \prod_{j=-(n-1)}^{n} q_{2j}^{3 - \alpha_{2j} - w_{2j}} \right) w \rm{,}
\end{align}
and so in particular
\begin{align*}
 \mathcal{E} \circ \uline{1212\ldots12} &= \sum_{w \in \mathcal{S}_E} \left( \prod_{j=-(n-1)}^{n} q_{2j}^{1 - w_{2j}} \right) w \rm{,} \\
 \mathcal{E} \circ \uline{2121\ldots21} &= \sum_{w \in \mathcal{S}_E} \left( \prod_{j=-(n-1)}^{n} q_{2j}^{2 - w_{2j}} \right) w \rm{.}
\end{align*}
Combining these observations, we see that for any $\alpha \in \mathcal{S}_0$
\begin{align*}
 \mathcal{E} \mathcal{O} \circ \alpha &= \mathcal{E} \circ \left( \sum_{v \in \mathcal{S}_O} \left( \prod_{j=-(n-1)}^{n} q_{2j-1}^{3 - \alpha_{2j-1} - v_{2j-1}} \right) v \right) \\
 &= \prod_{j=-(n-1)}^n q_{2j-1}^{2-\alpha_{2j-1}} \left( \mathcal{E} \circ \uline{1212\ldots12} \right)  \\
& \qquad + \prod_{j=-(n-1)}^n q_{2j-1}^{1-\alpha_{2j-1}} \left( \mathcal{E}  \circ \uline{2121\ldots21} \right) \\
 &= \prod_{j=-(n-1)}^n q_{2j-1}^{2-\alpha_{2j-1}} \sum_{w \in \mathcal{S}_E} \left( \prod_{j=-(n-1)}^{n} q_{2j}^{1 - w_{2j}} \right) w \\
& \qquad + \prod_{j=-(n-1)}^n q_{2j-1}^{1-\alpha_{2j-1}} \sum_{w \in \mathcal{S}_E} \left( \prod_{j=-(n-1)}^{n} q_{2j}^{2 - w_{2j}} \right) w  \rm{.}
\end{align*}
And so:
\begin{align*}
 \mathcal{O} \mathcal{E} \mathcal{O} \circ \alpha &= \prod_{j=-(n-1)}^n q_{2j-1}^{2-\alpha_{2j-1}} \mathcal{O} \circ \sum_{w \in \mathcal{S}_E} \left( \prod_{j=-(n-1)}^{n} q_{2j}^{1 - w_{2j}} \right) w \\
& \qquad + \prod_{j=-(n-1)}^n q_{2j-1}^{1-\alpha_{2j-1}} \mathcal{O} \circ \sum_{w \in \mathcal{S}_E} \left( \prod_{j=-(n-1)}^{n} q_{2j}^{2 - w_{2j}} \right) w  \\
 &= \prod_{j=-(n-1)}^n \frac{q_{2j-1}^{2-\alpha_{2j-1}}}{q_{2j}} \mathcal{O} \circ \uline{1212\ldots12} + \prod_{j=-(n-1)}^n q_{2j-1}^{2-\alpha_{2j-1}}  \mathcal{O}  \circ \uline{2121\ldots21} \\
& + \prod_{j=-(n-1)}^n q_{2j-1}^{1-\alpha_{2j-1}} \mathcal{O} \circ \uline{1212\ldots12} + \prod_{j=-(n-1)}^n q_{2j-1}^{1-\alpha_{2j-1}} q_{2j} \mathcal{O} \circ \uline{2121\ldots21}
\end{align*}
\begin{align*}
 &= \left( \prod_{j=-(n-1)}^n \frac{q_{2j-1}^{2-\alpha_{2j-1}}}{q_{2j}} + \prod_{j=-(n-1)}^n q_{2j-1}^{1-\alpha_{2j-1}} \right) \mathcal{O} \circ \uline{1212\ldots12} \\
& + \left( \prod_{j=-(n-1)}^n q_{2j-1}^{2-\alpha_{2j-1}} + \prod_{j=-(n-1)}^n q_{2j-1}^{1-\alpha_{2j-1}} q_{qj} \right) \mathcal{O} \circ \uline{2121\ldots21} \\
 &= \left( \prod_{j=-(n-1)}^n \frac{q_{2j-1}^{2-\alpha_{2j-1}}}{q_{2j}} + \prod_{j=-(n-1)}^n q_{2j-1}^{1-\alpha_{2j-1}} \right) \left( \sum_{v \in \mathcal{S}_O} \left( \prod_{j=-(n-1)}^{n} q_{2j-1}^{2 - v_{2j-1}} \right) v  \right) \\
& + \left( \prod_{j=-(n-1)}^n q_{2j-1}^{2-\alpha_{2j-1}} + \prod_{j=-(n-1)}^n q_{2j-1}^{1-\alpha_{2j-1}} q_{2j} \right) \left( \sum_{v \in \mathcal{S}_O} \left( \prod_{j=-(n-1)}^{n} q_{2j-1}^{1 - v_{2j-1}} \right) v \right) \\
&= \left( \prod_{j=-(n-1)}^n \frac{q_{2j-1}}{q_{2j}} + 1 \right) \left( \sum_{v \in \mathcal{S}_O} \left( \prod_{j=-(n-1)}^{n} q_{2j-1}^{3 - \alpha_{2j-1} - v_{2j-1}} \right) v  \right) \\
& + \left( 1 + \prod_{j=-(n-1)}^n \frac{q_{2j}}{q_{2j-1}} \right) \left( \sum_{v \in \mathcal{S}_O} \left( \prod_{j=-(n-1)}^{n} q_{2j-1}^{3 - \alpha_{2j-1} - v_{2j-1}} \right) v \right) \\
\mathcal{O}\mathcal{E}\mathcal{O} &= \left( \prod_{j=-(n-1)}^n \frac{q_{2j-1}}{q_{2j}} + 2 + \prod_{j=-(n-1)}^n \frac{q_{2j}}{q_{2j-1}} \right) \left( \left( \prod_{j=-(n-1)}^{n} q_{2j-1}^{3 - \alpha_{2j-1} - v_{2j-1}} \right) v \right) \\
&= Q \mathcal{O} \circ \alpha
\end{align*}
which was the claimed result.

The second part of the lemma follows from a very similar calculation.
\end{proof}

The identity $\mathcal{O} \mathcal{E} \mathcal{O} = Q \mathcal{O}$ is illustrated for $n=2$ in Figure \ref{fig:indexlemma2}.

\begin{figure}
 \centering
 \scalebox{0.3}{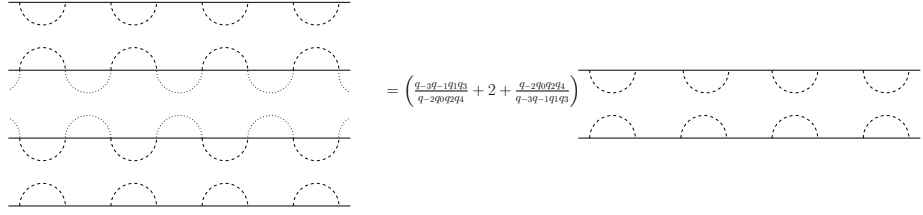}
 \caption{$\mathcal{O} \mathcal{E} \mathcal{O} = Q \mathcal{O}$ for $n=2$}
 \label{fig:indexlemma2}
\end{figure}

We can now prove Theorem \ref{thm:rep}.

\begin{proof}[Proof of Theorem \ref{thm:rep}]
We have to show that the map $\mathcal{R}$ preserves the algebra relations given in Definition \ref{def:sbalgebra} if and only if the parameters $a,b,c,d,x,y,z,w$ satisfy conditions \eqref{thmcon1} through \eqref{thmcon6}.

First we note that for any invertible $q,r$ the operators $R^q_i$ and $R^{r}_j$ commute when $|i-j| \neq 1$.  Therefore $\mathcal{R}$ always preserves relations \eqref{rel:sba8}, \ref{rel:sba10}, \eqref{rel:sba11} and \eqref{rel:sba12} (whatever the specific action $\theta$).

By Lemma \ref{lem:TLlike} we have that, for any invertible $q \in \mathcal{A}$, $(R^{q}_{i})^2 = [2]_q R^q_i$.  This, together with the observation above, suffices to show that 
\begin{align*}
 \mathcal{R}(U_i) \mathcal{R}(U_i) &= \left( R^a_{-n-i} R^b_{-n+i} R^c_{n-i} R^d_{n+i} \right) \left( R^a_{-n-i} R^b_{-n+i} R^c_{n-i} R^d_{n+i} \right) \\
 &= \left( R^a_{-n-i} \right)^2 \left( R^b_{-n+i} \right)^2 \left( R^c_{n-i} \right)^2 \left( R^d_{n+i} \right)^2 \\
 &= [2]_a [2]_b [2]_c [2]_d \mathcal{R}(U_i) \rm{.}
\end{align*}
Hence $\mathcal{R}$ preserves relation \eqref{rel:sba1} if and only if $\delta = [2]_a [2]_b [2]_c [2]_d$ - that is, if and only if condition \eqref{thmcon1} holds.

Similar calculations show that $\mathcal{R}$ preserves relation \eqref{rel:sba2} if and only if condition \eqref{thmcon2} holds, and that $\mathcal{R}$ preserves relation \eqref{rel:sba3} if and only if condition \eqref{thmcon3} holds.

To check relations $\eqref{rel:sba4}$ and $\eqref{rel:sba5}$, we can use Lemma \ref{lem:1} to show that:
\begin{align*}
 \mathcal{R}(U_1) \mathcal{R}(e) \mathcal{R}(U_1) &= \left( R^a_{-n-1} R^b_{-n+1} R^c_{n-1} R^d_{n+1} \right) \left( R^x_{-n} R^y_n \right) \left( R^a_{-n-1} R^b_{-n+1} R^c_{n-1} R^d_{n+1} \right) \\
 &= \left( R^a_{-n-1} R^b_{-n+1} R^x_{-n} R^a_{-n-1} R^b_{-n+1} \right) \left( R^c_{n-1} R^d_{n+1} R^y_{n} R^c_{n-1} R^d_{n+1} \right) \\
 &= \left( [2]_{\frac{x}{ab}} R^a_{-n-1} R^b_{-n+1} \right) \left( [2]_{\frac{y}{cd}} R^c_{n-1} R^d_{n+1} \right) \\
&= \left( \frac{ab}{x} + \frac{x}{ab} \right) \left( \frac{cd}{y} + \frac{y}{cd} \right) R^a_{-n-1} R^b_{-n+1} R^c_{n-1} R^d_{n+1} 
\end{align*}
and
\begin{align*}
 \mathcal{R}(U_{n-1}) \mathcal{R}(f) \mathcal{R}(U_{n-1}) &= \left( R^a_{-2n+1} R^b_{-1} R^c_{1} R^d_{2n-1} \right) \left( R^z_0 R^w_{2n} \right) \left( R^a_{-2n+1} R^b_{-1} R^c_{1} R^d_{2n-1} \right) \\
 &= \left( R^b_{-1} R^c_1 R^z_0 R^c_1 R^b_{-1} \right) \left( R^a_{-2n+1} R^d_{2n-1} R^w_{2n} R^a_{-2n+1} R^d_{2n-1} \right) \\
 &= \left( [2]_{\frac{z}{cb}} R^b_{-1} R^c_1 \right) \left( [2]_{\frac{w}{ab}} R^a_{-2n+1} R^d_{2n-1} \right) \\
 &= \left( \frac{cd}{z} + \frac{z}{cd} \right) \left( \frac{ad}{w} + \frac{w}{ad} \right) R^a_{-2n+1} R^b_{-1} R^c_{1} R^d_{2n-1} \rm{.}
\end{align*}
So $\mathcal{R}$ preserves relation \eqref{rel:sba4} if and only if condition \eqref{thmcon4} holds, and $\mathcal{R}$ preserves relation \eqref{rel:sba5} if and only if condition \eqref{thmcon5} holds.

It remains to check relations $\eqref{rel:sba9}$, $\eqref{rel:sba6}$ and $\eqref{rel:sba7}$.

For the first of these, we can again use Lemma \ref{lem:TLlike}.  For $|i-j|=1$, we have
\begin{align*}
 \mathcal{R}(U_i) \mathcal{R} (U_j) \mathcal{R} (U_i) &= \left( R^a_{-n-i} R^b_{-n+i} R^c_{n-i} R^d_{n+i} \right) \left( R^a_{-n-j} R^b_{-n+j} R^c_{n-j} R^d_{n+j} \right) \\
 & \qquad \qquad \times \left( R^a_{-n-i} R^b_{-n+i} R^c_{n-i} R^d_{n+i} \right) \\
 &= \left( R^a_{-n-i} R^a_{-n-j} R^a_{-n-i} \right) \left( R^b_{-n+i} R^a_{-n+j} R^b_{-n+i} \right) \\
& \qquad \qquad \times \left( R^c_{n-i} R^c_{n-j} R^c_{n-i} \right) \left( R^d_{n+i} R^d_{n+j} R^d_{n+i} \right) \\
 &= R^a_{-n-i} R^b_{-n+i} R^c_{n-i} R^d_{n+i} \\
 &= \mathcal{R}(U_i) \rm{.}
\end{align*}
So, in fact, $\mathcal{R}$ preserves relation \eqref{rel:sba9} whatever the algebra action on $\mathcal{A}$.

Finally, for relations \eqref{rel:sba6} and \eqref{rel:sba7} we use Lemma \ref{lem:IJ}. 

For $i \in I_n, $, let $\{q_i\}$ be given by
\begin{align*}
 q_i &= \left\{ \begin{array}{ll}
                 a & \mbox{if } -2n+1 \leq i < -n \\
                 x & \mbox{if } i = -n \\
                 b & \mbox{if } -n < -i < 0 \\
                 z & \mbox{if } i = 0 \\
                 c & \mbox{if } 0 < i < n \\
                 y & \mbox{if } i = n \\
                 d & \mbox{if } n < i < 2n \\
                 w & \mbox{if } i = 2n
                \end{array}
 \right. \rm{.}
\end{align*}
Then for $n$ even, $\mathcal{O} = \mathcal{R}(I)$ and $\mathcal{E} = \mathcal{R}(J)$, while for $n$ odd $\mathcal{O} = \mathcal{R}(J)$ and $\mathcal{E} = \mathcal{R}(I)$.

If $n$ is even, we now have
\begin{align*}
 Q &= \frac{a^{\frac{n}{2}-1} b^{\frac{n}{2}-1} c^{\frac{n}{2}-1} d^{\frac{n}{2}-1}}{a^{\frac{n}{2}-2} x b^{\frac{n}{2}-2} z c^{\frac{n}{2}-2} y d^{\frac{n}{2}-2} w} + 2 + \frac{a^{\frac{n}{2}-2} x b^{\frac{n}{2}-2} z c^{\frac{n}{2}-2} y d^{\frac{n}{2}-2} w}{a^{\frac{n}{2}-1} b^{\frac{n}{2}-1} c^{\frac{n}{2}-1} d^{\frac{n}{2}-1}} \\
 &= \frac{abcd}{xyzw} + 2 + \frac{xyzw}{abcd}  \rm{,}
\end{align*}
while if $n$ is odd, we have
\begin{align*}
 Q &= \frac{a^\frac{n-1}{2} x b^\frac{n-1}{2} c^\frac{n-1}{2} y d^\frac{n-1}{2}}{a^\frac{n-1}{2} b^\frac{n-1}{2} z c^\frac{n-1}{2} d^\frac{n-1}{2} w} + 2 + \frac{{a^\frac{n-1}{2} b^\frac{n-1}{2} z c^\frac{n-1}{2} d^\frac{n-1}{2} w}}{a^\frac{n-1}{2} x b^\frac{n-1}{2} c^\frac{n-1}{2} y d^\frac{n-1}{2}} \\
 &= \frac{xy}{zw} + 2 + \frac{zw}{xy}  \rm{.}
\end{align*}

Now note that for any $n \in \mathbb{N}$ we have
\begin{align*}
 \mathcal{R}(I)\mathcal{R}(J)\mathcal{R}(I) &= Q \mathcal{R}(I) \rm{,} \\
 \mathcal{R}(J)\mathcal{R}(I)\mathcal{R}(J) &= Q \mathcal{R}(J) \rm{.}
\end{align*}

Hence, relations \eqref{rel:sba6} and \eqref{rel:sba7} hold if and only if $Q = \kappa$.  That is, if and only if condition \eqref{thmcon6} holds. 

This completes the proof of the theorem. \end{proof}

\section{Specialisations to a field $k$}

We have now constructed a (unique) tensor space representation $\mathcal{R}$ of $\sbaA$ over the ring $\mathcal{A}$, acting on the $\mathcal{A}$-module $V^{\otimes 4n}$.  However, as discussed in the introduction, we are really interested in representations of the symplectic blob algebra over a field, $\sba := k \otimes B_n^{\prime}$.  

Fix six arbitrary elements $\{ \delta, \delta_L, \delta_R, \kappa_L, \kappa_R, \kappa \} \subset k$. There is a unique unital ring homomorphism $\rho : \mathcal{Z} \rightarrow k$ satisfying $\rho(D) = \delta$, $\rho(D_L) = \delta_L$, $\rho(D_R) = \delta_R$, $\rho(K_L) = \kappa_L$, $\rho(K_R) = \kappa_R$ and $\rho(K) = \kappa$; this map makes $k$ into a $\mathcal{Z}$-algebra.  For $k$ such an algebra, define
\begin{displaymath}
 \sba(\uline{\Pi}) := k \otimes_{\mathcal{Z}} B_n^{\prime} \rm{.}
\end{displaymath}
 where $\uline{\Pi} := (\delta, \delta_L, \delta_R, \kappa_L, \kappa_R, \kappa)$.

In this section we show that, for any choice of $\uline{\Pi} \in k^{6}$, we may define a map $\iota : \mathcal{A} \rightarrow k$, making $k$ into an $\mathcal{A}$-algebra, such that 
\begin{displaymath}
 \mathcal{V}(n) := k \otimes_{\mathcal{A}} V^{\otimes 4n}
\end{displaymath}
is a $\sba(\uline{\Pi})$-module.

Let $\{ a_0,b_0,c_0,d_0,x_0,y_0,z_0,w_0 \} \subset k$, and define $\uline{\Sigma} := (a_0, b_0, c_0, d_0, x_0, y_0, z_0, w_0)$.  Then there is a unique (unital) ring homomorphism, $\iota : \mathcal{A} \rightarrow k$, which sends $a$ to $a_0$, $b$ to $b_0$ and so on.  This map makes $k$ into an $\mathcal{A}$-algebra.  As $\mathcal{A}$ is itself a $\mathcal{Z}$-algebra, this also makes $k$ into a $\mathcal{Z}$-algebra (via the map $\iota \circ \theta : \mathcal{Z} \rightarrow k$).  So we can define a symplectic blob algebra over this field, as $k \otimes_{\mathcal{A}} \sbaA$.

Suppose that $\uline{\Pi}$ and $\uline{\Sigma}$ are chosen such that $\rho = \iota \circ \theta$; that is, suppose that the diagram below commutes.

\begin{displaymath}
 \xymatrix{ \mathcal{Z} \ar[rr]^{\theta} \ar[dr]_{\rho} & & \mathcal{A} \ar[dl]^{\iota} \\
& k &
}
\end{displaymath}

In this case, the algebra $\sbak(\uline{\Pi}) = k \otimes_{\mathcal{A}} \sbaA$ and acts on the $k$-module $k \otimes_{\mathcal{A}} V^{\otimes 4n}$.  We have the following result:

\begin{prop}  \label{prop:repoverk}
 Suppose that $\rho = \iota \circ \theta$.  Let $\mathcal{V}(n) := k \otimes_{\mathcal{A}} V^{\otimes 4n}$.  Define a map $\mathcal{R}_{\uline{\Sigma}} : \sbak \rightarrow \End_k(\mathcal{V}(n))$ by
\begin{displaymath}
 \mathcal{R}_{\uline{\Sigma}} := \iden_k \otimes_{\mathcal{A}} \mathcal{R} \rm{,}
\end{displaymath}
where $\mathcal{R}$ is the representation of $\sbaA$ defined in Theorem \ref{thm:rep}.

Then $\mathcal{R}_{\uline{\Sigma}}$ is a representation of $\sbak(\uline{\Pi})$.
\end{prop}
\begin{proof}
 In light of the discussion above, this is an immediate corollary of Theorem \ref{thm:rep}.
\end{proof}

The discussion above suggests that we must be careful in our choice of both $\uline{\Pi}$ and $\uline{\Sigma}$.  In fact, however, we will now show that for \emph{any} arbitrary choice of $\uline{\Pi}$ there exists a $\uline{\Sigma}$ such that $\rho = \iota \circ \theta$.  That is, we can construct a tensor space module $\mathcal{V}(n)$ for any specialisation $\sbak(\uline{\Pi})$ of the symplectic blob algebra over an algebraically closed field $k$.

\begin{prop}   \label{corl:repexists}
 Suppose that $k$ is an algebraically closed field.  Fix $\uline{\Pi} \in k^6$.   Then there exists some $\uline{\Sigma} = (a_0, b_0, c_0, d_0, x_0, y_0, z_0, w_0) \in k^{8}$ such that $\rho = \iota \circ \theta$ and $\mathcal{R}_{\uline{\Sigma}}$ is a representation of $\sbak(\uline{\Pi})$.
\end{prop}
\begin{proof}
Consider the image of relations (\ref{thmcon1}) through (\ref{thmcon6}) under the map $\iota$.  We want to show that, for $\uline{\Pi}$ fixed but arbitrary, there must exist some $\uline{\Sigma}$ such that:
\begin{align}
  \delta &= \left( a_0 + \frac{1}{a_0} \right) \left( b_0 + \frac{1}{b_0} \right) \left( c_0 + \frac{1}{c_0} \right) \left( d_0 + \frac{1}{d_0} \right) \label{corlcon1} \\
  \delta_L &= \left( x_0 + \frac{1}{x_0} \right) \left( y_0 + \frac{1}{y_0} \right) \label{corlcon2} \\
  \delta_R &= \left( z_0 + \frac{1}{z_0} \right) \left( w_0 + \frac{1}{w_0} \right) \label{corlcon3} \\
  \kappa_L &= \left( \frac{a_0 b_0}{x_0} + \frac{x_0}{a_0 b_0} \right) \left( \frac{c_0 d_0}{y_0} + \frac{y_0}{c_0 d_0} \right) \label{corlcon4} \\
  \kappa_R &= \left( \frac{a_0 d_0}{w_0} + \frac{w_0}{a_0 d_0} \right) \left( \frac{b_0 c_0}{z_0} + \frac{z_0}{b_0 c_0} \right) \label{corlcon5} \\
  \kappa &= \frac{x_0 y_0}{z_0 w_0} + 2 + \frac{z_0 w_0 }{x_0 y_0} \label{corlcon6a} \\
  \kappa &= \frac{a_0 b_0 c_0 d_0}{x_0 y_0 z_0 w_0} + 2 + \frac{x_0 y_0 z_0 w_0}{a_0 b_0 c_0 d_0} \rm{.} \label{corlcon6b}
 \end{align}
(Note that this is a slightly stronger result than is needed -- it would suffice to show that either one of the last two conditions is satisfied, depending on the polarity of $n$ -- but in fact the stronger result holds and the resulting proof is slightly more compact.)

Our approach is to rewrite the conditions above in the form $F_1(X_1) = 0$, $F_2(X_1, X_2) = 0$, $F_3(X_1, X_2, X_3) = 0$, and so on, where each $F_i$ is a polynomial in the indeterminates $X_1, \ldots X_i$.  Since $k$ is algebraically closed, we can solve the first equation for $X_1$, then solve the second for $X_2$, and so on, at each stage checking that the solution obtained is invertible (that is, non-zero).

Let $E := ABCD$, $F := \frac{AB}{CD}$ and $G = \frac{AD}{BC}$.  Note that given any tuple $(D,E,F,G) \in k^4$ we can recover $(A,B,C,D)$.  Similarly, let $P = ZW$ and $Q = \frac{Z}{W}$.  Given any $(P,Q)$ we can recover $(Z,W)$.

Now let $x_0$ be any non-zero element of $k$.  Condition (\ref{corlcon2}) now states that $y_0$ is a solution of the following equation in $Y$:
\begin{equation*}
  \delta_L = \left( x_0 + \frac{1}{x_0} \right) \left( Y + \frac{1}{Y} \right) \rm{,}
\end{equation*}
or
\begin{equation}  \label{eq:corlsolv1}
 Y^2 - \left( \frac{x_0 \delta_L}{x_0^2 + 1} \right) Y + 1 = 0 \rm{.}
\end{equation}
This equation has (at least) one solution since $k$ is algebraically closed, and clearly $Y=0$ is not a solution.  So let $y_0$ be any solution of (\ref{eq:corlsolv1}).

We fix $e_0 := a_0 b_0 c_0 d_0$ such that conditions (\ref{corlcon6a}) and (\ref{corlcon6b}) are equivalent, by setting $e_0 = x_0^2 y_0^2$.  Clearly $e_0$ is non-zero.

Condition (\ref{corlcon6a}) (or equivalently (\ref{corlcon6b})) now states that $p_0 := z_0 w_0$ must be a solution of the following equation in $P$:
\begin{equation*}
 \kappa = \frac{x_0 y_0}{P} + 2 + \frac{P}{x_0 y_0} \rm{,}
\end{equation*}
or
\begin{equation} \label{eq:corlsolv3}
P^2 + \left( 2 - \kappa \right) x_0 y_0 P + e_0 = 0 \rm{.} 
\end{equation}
Clearly at least one solution of this equation exists, and $P=0$ is not a solution.  So let $p_0$ be any solution of (\ref{eq:corlsolv3}).

Condition (\ref{corlcon3}) now states that $q_0 = \frac{z_0}{w_0}$ must be a solution of the following equation in $Q$:
\begin{equation*}
\delta_R = p_0 + \frac{1}{p_0} + Q + \frac{1}{Q}  \rm{,}
\end{equation*}
or
\begin{equation}  \label{eq:corlsolv4}
 Q^2 + (p_0 + \frac{1}{p_0} - \delta_R) Q + 1 = 0 \rm{.}
\end{equation}
Once again, since $k$ is algebraically closed this equation must have at least one solution, and it is clear that $Q=0$ is not a solution.  So let $q_0$ be any solution of $(\ref{eq:corlsolv4})$.

Now condition (\ref{corlcon4}) states that $f_0 := \frac{a_0 b_0}{c_0 d_0}$ must be a solution of the following equation in $F$:
\begin{equation*}
 \kappa_L = \frac{Fy_0}{x_0} + \frac{x_0 y_0}{e_0} + \frac{e_0}{x_0 y_0} + \frac{x_0}{Fy_0} \rm{,}
\end{equation*}
or
\begin{equation}  \label{eq:corlsolv5}
 F^2 + \left( x_0^2 + \frac{1}{y_0^2} - \frac{x_0}{y_0} \kappa_L \right) F + \frac{x_0^2}{y_0^2} = 0\rm{.}
\end{equation}
Clearly at least one such solution exists, and $F=0$ is not a solution.  So let $f_0$ be a solution of (\ref{eq:corlsolv5}).

Similarly, condition (\ref{corlcon5}) states that $g_0 := \frac{a_0 d_0}{b_0 c_0}$ must be a solution of the following equation in $G$:
\begin{equation*}
\kappa_R = G_0 q_0 +  \frac{p_0}{e_0} + \frac{e_0}{p_0} + \frac{1}{G q_0}  \rm{,}
\end{equation*}
or
\begin{equation}  \label{eq:corlsolv6}
G^2 + \left( \frac{e_0}{p_0 q_0} + \frac{p_0}{e_0 q_0} - \frac{\kappa_L}{q_0} \right) G + \frac{1}{q_0^2} = 0 \rm{.}
\end{equation}
Clearly, this equation has a solution (since $k$ is algebraically closed) and $G=0$ is not such a solution.  So let $g_0$ be any solution of (\ref{eq:corlsolv6}).

Finally, condition (\ref{corlcon1}) states -- after some simplification -- that $d_0$ must be a solution of the following equation in $D$:
\begin{align}  \label{eq:corlsolv7}
\frac{f_0}{e_0 g_0} D^8 + \left( f_0 + \frac{1}{e_0} + \frac{1}{g_0} + \frac{f_0}{e_0 g_0} \right) D^6 \nonumber \\  \qquad   + \left( e_0 + f+0 + g_0 + \frac{1}{g_0} + \frac{1}{f_0} + \frac{1}{e_0} \right) D^4 \\ \qquad + \left( \frac{1}{f_0} + e_0 + g_0 + \frac{e_0 g_0}{f_0} \right) D^2 + \frac{e_0 g_0}{f_0} = 0 \nonumber \rm{.} 
\end{align}
Since $k$ is algebraically closed, this equation must have (at least) one solution, and it is not difficult to see that $D=0$ is not a solution.  So let $d_0$ be any solution of (\ref{eq:corlsolv7}).

We have now shown that, given an arbitrary $\uline{\Pi} = (\delta, \delta_L, \delta_R, \kappa_L, \kappa_R, \kappa) \in k^{6}$ and some non-zero $x_0 \in k$ there exist $y_0, e_0, p_0, q_0, f_0, g_0, d_0 \in k$ such that the $8$-tuple
\begin{displaymath}
 \uline{\Sigma} := \left( \sqrt{ \frac{e_0 g_0}{d_0^2} }, \sqrt{ \frac{f_0 d_0^2}{g_0} }, \sqrt{ \frac{e_0}{f_0 d_0^2} }, d_0, x_0, y_0, \sqrt{p_0 q_0}, \sqrt{ \frac{p_0}{q_0} }   \right)  \rm{,}
\end{displaymath}
makes $\rho = \iota \circ \theta$ and so satisfies the conditions of Proposition \ref{prop:repoverk}.  This completes the proof.
\end{proof}

\section{Discussion}

We have now constructed, for any $n \in \mathbb{N}$, a tensor space $\sbaA$ module, $V^{\otimes 4n}$  For $k$ an algebraically closed field of characteristic $0$, and $\uline{\Pi}$ any fixed (but arbitrary) $6$-tuple of parameters $(\delta, \delta_L, \delta_R, \kappa_L, \kappa_R, \kappa) \in k^{6}$ we have shown that this module passes to a $\sba(\uline{\Pi})$-module, $\mathcal{V}(n)$.

In our next paper, we will consider \emph{quasihereditary} specialisations of $\sba(\uline{\Pi})$.  The symplectic blob algebra is quasihereditary whenever each of the six parameters $\delta$, $\delta_L$, $\delta_R$, $\kappa_L$, $\kappa_R$ and $\kappa$ are non-zero \cite{greenmartinparker2007}.

We will show that, for these specialisations, the tensor space $\mathcal{A}$-module $V$ passes to a \emph{full-tilting} $k$-module $\mathcal{V}(n)$.  The construction of such a module allows for the calculation of the Ringel dual $\End_{\sba}(\mathcal{V}(n))^{\op}$ and of the (indecomposable) tilting modules $T_n(\lambda)$, which must occur as direct summands of $\mathcal{V}(n)$.  It may therefore be used to study the non-generic representation theory of $\sba$, which is -- at present -- not well understood \cite{greenmartinparker2008}.

\subsection*{Acknowledgments}

While writing this paper the author received funding from an EPSRC training grant.  He would like to thank his supervisor, Paul Martin, for introducing him to this field and suggesting this problem, and for many useful discussions and suggestions.

The author would also like to thank Elizabeth Banjo and Alison Parker, both of whom read and commented on previous drafts of this paper.

\bibliography{tilting}
\bibliographystyle{amsplain}

\end{document}